\documentclass[oneside,english]{amsart}
\usepackage[T1]{fontenc}
\usepackage[latin9]{inputenc}
\usepackage{geometry}
\usepackage{mathtools}
\usepackage{amstext}
\usepackage{amsthm}
\usepackage{amssymb}

\makeatletter
\numberwithin{equation}{section}
\numberwithin{figure}{section}
\theoremstyle{plain}
\newtheorem{thm}{\protect\theoremname}
\theoremstyle{remark}
\newtheorem{rem}[thm]{\protect\remarkname}
\theoremstyle{definition}
\newtheorem{example}[thm]{\protect\examplename}
\theoremstyle{plain}
\newtheorem{lem}[thm]{\protect\lemmaname}
\theoremstyle{plain}
\newtheorem{prop}[thm]{\protect\propositionname}
\theoremstyle{plain}
\newtheorem{cor}[thm]{\protect\corollaryname}

\usepackage{shuffle}
\usepackage[all]{xy}

\makeatother

\usepackage{babel}
\providecommand{\corollaryname}{Corollary}
\providecommand{\examplename}{Example}
\providecommand{\lemmaname}{Lemma}
\providecommand{\propositionname}{Proposition}
\providecommand{\remarkname}{Remark}
\providecommand{\theoremname}{Theorem}

\begin{document}
\address[Minoru Hirose]{Institute for Advanced Research, Nagoya University,  Furo-cho, Chikusa-ku, Nagoya, 464-8602, Japan}
\email{minoru.hirose@math.nagoya-u.ac.jp}
\address[Hideki Murahara]{The University of Kitakyushu,  4-2-1 Kitagata, Kokuraminami-ku, Kitakyushu, Fukuoka, 802-8577, Japan}
\email{hmurahara@mathformula.page}
\address[Shingo Saito]{Faculty of Arts and Science, Kyushu University,  744, Motooka, Nishi-ku, Fukuoka, 819-0395, Japan}
\email{ssaito@artsci.kyushu-u.ac.jp}

\global\long\def\mstuffle{\mathbin{{\tilde{*}}}}%

\title{Ohno relation for regularized multiple zeta values}
\author{Minoru Hirose}
\author{Hideki Murahara}
\author{Shingo Saito}
\begin{abstract}
The Ohno relation for multiple zeta values can be formulated as saying
that a certain operator, defined for indices, is invariant under taking
duals. In this paper, we generalize the Ohno relation to regularized
multiple zeta values by showing that, although the suitably generalized
operator is not invariant under taking duals, the relation between
its values at an index and at its dual index can be written explicitly
in terms of the gamma function.
\end{abstract}

\maketitle

\section{Introduction}

Multiple zeta values (MZVs) are real numbers defined by
\[
\zeta(k_{1},\dots,k_{d})\coloneqq\sum_{0<m_{1}<\cdots<m_{d}}\frac{1}{m_{1}^{k_{1}}\cdots m_{d}^{k_{d}}},
\]
where $k_{1},\dots,k_{d}$ are positive integers with $k_{d}>1$.
A lot of formulas for MZVs are known and the Ohno relation is one
of the most famous among such formulas. Put $\mathfrak{h}\coloneqq\mathbb{Q}\left\langle x,y\right\rangle $
and $\mathfrak{h}^{0}\coloneqq\mathbb{Q}\oplus y\mathfrak{h}x$. Define
a linear map $Z:\mathfrak{h}^{0}\to\mathbb{R}$, a ring homomorphism
$\sigma:\mathfrak{h}\to\mathfrak{h}[[T]]$, and an anti-automorphism
$\tau:\mathfrak{h}\to\mathfrak{h}$ by
\[
Z(yx^{k_{1}-1}\cdots yx^{k_{d}-1})=\zeta(k_{1},\dots,k_{d}),
\]
\[
\sigma(x)=x,\sigma(y)=y\frac{1}{1-xT},
\]
and
\[
\tau(x)=y,\tau(y)=x.
\]
In this paper, we use $T,A,B$ as parameters of formal power series,
and we regard any linear map $g$ on a $\mathbb{Q}$-vector space
$R$ as naturally extended to the linear map $g$ on $R\otimes\mathbb{Q}[[T,A,B]]$
by $\sum_{t,a,b\geq0}g(c_{t,a,b}T^{t}A^{a}B^{b})=\sum_{t,a,b\geq0}g(c_{t,a,b})T^{t}A^{a}B^{b}$.
Then $Z\circ\sigma(w)$ is defined for any $w\in\mathfrak{h}^{0}$
as an element in $\mathbb{R}_{T}\coloneqq\mathbb{R}[[T]]$.
\begin{thm}[Ohno relation for MZVs, \cite{Ohno}]
\label{thm:OhnoRel}For $w\in\mathfrak{h}^{0}$, we have
\[
Z\circ\sigma\circ\tau(w)=Z\circ\sigma(w).
\]
\end{thm}

\begin{rem}
The original formulation of the Ohno relation in \cite{Ohno} is described
as 
\[
\sum_{\substack{e_{1}+\cdots+e_{r}=m\\
e_{1},\dots,e_{r}\geq0
}
}\zeta(k_{1}+e_{1},\dots,k_{r}+e_{r})=\sum_{\substack{e_{1}+\cdots+e_{s}=m\\
e_{1},\dots,e_{s}\geq0
}
}\zeta(k_{1}'+e_{1},\dots,k_{s}'+e_{s}),
\]
where $(k_{1}',\dots,k_{s}')$ is the dual index of $(k_{1},\dots,k_{r})$
i.e., $\tau(yx^{k_{1}-1}\cdots yx^{k_{r}-1})=yx^{k_{1}'-1}\cdots yx^{k_{s}'-1}$,
and it corresponds to the equality of the coefficient of $T^{m}$
in Theorem \ref{thm:OhnoRel} in the case $w=yx^{k_{1}-1}\cdots yx^{k_{d}-1}$.
A formulation similar to the above theorem can be found in \cite{IKZ}.
\end{rem}

The shuffle product $\shuffle:\mathfrak{h}\times\mathfrak{h}\to\mathfrak{h}$
is the bilinear map defined by 
\[
w\shuffle1=1\shuffle w=w\ \ \ (w\in\mathfrak{h}),
\]
\[
u_{1}w_{1}\shuffle u_{2}w_{2}=u_{1}(w_{1}\shuffle u_{2}w_{2})+u_{2}(u_{1}w_{1}\shuffle w_{2})\ \ (u_{i}\in\{x,y\},w_{i}\in\mathfrak{h}).
\]
It is known that $Z(w_{1}\shuffle w_{2})=Z(w_{1})Z(w_{2})$ for $w_{1},w_{2}\in\mathfrak{h}^{0}$.
Let $Z_{X,Y}^{\shuffle}:\mathfrak{h}\to\mathbb{R}[X,Y]$ be the unique
map characterized by $\left.Z_{X,Y}^{\shuffle}\right|_{\mathfrak{h}^{0}}=Z$,
$Z_{X,Y}^{\shuffle}(w_{1}\shuffle w_{2})=Z_{X,Y}^{\shuffle}(w_{1})Z_{X,Y}^{\shuffle}(w_{2})$
for $w_{1},w_{2}\in\mathfrak{h}$, $Z_{X,Y}^{\shuffle}(x)=X$, and
$Z_{X,Y}^{\shuffle}(y)=Y$. We put $Z^{\shuffle}=Z_{0,0}^{\shuffle}$.
Then for any $w\in\mathfrak{h}$, $Z_{X,Y}^{\shuffle}(w)$ becomes
a polynomial of $X$ and $Y$ whose coefficients are linear combinations
of MZVs, and called a (shuffle) regularized MZV. The term `shuffle
regularized MZVs' often means the values $\zeta^{\shuffle}(k_{1},\dots,k_{d})=Z_{0,0}^{\shuffle}(yx^{k_{1}-1}\cdots yx^{k_{d}-1})$
with $k_{1},\dots,k_{d}\geq1$, namely the values for words starting
with $y$, but in this paper, we also consider the values for words
starting with $x$. The main theorem of this paper is a generalization
of Theorem \ref{thm:OhnoRel} to regularized MZVs. To state the main
theorem, define an $\mathbb{R}$-linear map $\rho:\mathbb{R}[X,Y]\to\mathbb{R}_{T}[X,Y]$
by the equality
\[
\rho(e^{AX+BY})=\frac{\Gamma(1+A)\Gamma(1-T+B)}{\Gamma(1+B)\Gamma(1-T+A)}e^{AX+BY}
\]
 in $\mathbb{R}_{T}[X,Y][[A,B]]$. Equivalently, $\rho$ is determined
by
\[
\rho\left(\frac{X^{k}Y^{l}}{k!l!}\right)=\sum_{k'=0}^{k}\sum_{l'=0}^{l}c(k-k',l-l')\frac{X^{k'}Y^{l'}}{k'!l'!},
\]
where the coefficients $c(-,-)\in\mathbb{R}_{T}$ are given by
\[
\sum_{k=0}^{\infty}\sum_{l=0}^{\infty}c(k,l)A^{k}B^{l}=\frac{\Gamma(1+A)\Gamma(1-T+B)}{\Gamma(1+B)\Gamma(1-T+A)}.
\]

\begin{thm}[Main theorem; Ohno relation for regularized MZVs]
\label{thm:MainThm}For $w\in\mathfrak{h}$, we have
\[
Z_{Y,X}^{\shuffle}\circ\sigma\circ\tau(w)=\rho\circ Z_{X,Y}^{\shuffle}\circ\sigma(w).
\]
\end{thm}

\begin{example}
Let $(k_{1},\dots,k_{d})\in\mathbb{Z}_{\geq1}^{d}$ be an admissible
index (i.e., $k_{d}>1$) and $(k_{1}',\dots,k_{r}')$ its dual index.
Let us look at the case $w=xyx^{k_{1}-1}\cdots yx^{k_{d}-1}$ in Theorem
\ref{thm:MainThm}. By using the notation $\zeta_{l}(k_{1},\dots,k_{d})\coloneqq Z^{\shuffle}(x^{l}yx^{k_{1}-1}\cdots yx^{k_{d}-1})$,
we have
\begin{align*}
Z_{Y,X}^{\shuffle}\circ\sigma\circ\tau(w) & =\sum_{m=0}^{\infty}T^{m}\sum_{\substack{\substack{e_{1}+\cdots+e_{r+1}=m\\
e_{1},\dots,e_{r+1}\geq0
}
}
}\zeta_{0}(k_{1}'+e_{1},\dots,k_{r}'+e_{r},1+e_{r+1})\\
 & \ +X\sum_{m=0}^{\infty}T^{m}\sum_{\substack{\substack{e_{1}+\cdots+e_{r}=m\\
e_{1},\dots,e_{r}\geq0
}
}
}\zeta(k_{1}'+e_{1},\dots,k_{r}'+e_{r})
\end{align*}
and
\begin{align*}
Z_{X,Y}^{\shuffle}\circ\sigma(w) & =\sum_{m=0}^{\infty}T^{m}\sum_{\substack{\substack{e_{1}+\cdots+e_{d}=m\\
e_{1},\dots,e_{d}\geq0
}
}
}\zeta_{1}(k_{1}+e_{1},\dots,k_{d}+e_{d})\\
 & \ +X\sum_{m=0}^{\infty}T^{m}\sum_{\substack{\substack{e_{1}+\cdots+e_{d}=m\\
e_{1},\dots,e_{d}\geq0
}
}
}\zeta(k_{1}+e_{1},\dots,k_{d}+e_{d}).
\end{align*}
Since $\rho(X)=X+\sum_{n=1}^{\infty}\zeta(n+1)T^{n}$, the equality
of the coefficients of $T^{m}X^{1}Y^{0}$ implies the original Ohno
relation, and the equality of the coefficients of $T^{m}X^{0}Y^{0}$
implies
\begin{align*}
 & \sum_{\substack{\substack{e_{1}+\cdots+e_{r+1}=m\\
e_{1},\dots,e_{r+1}\geq0
}
}
}\zeta_{0}(k_{1}'+e_{1},\dots,k_{r}'+e_{r},1+e_{r+1})\\
 & =\sum_{\substack{\substack{e_{1}+\cdots+e_{d}=m\\
e_{1},\dots,e_{d}\geq0
}
}
}\zeta_{1}(k_{1}+e_{1},\dots,k_{d}+e_{d})+\sum_{\substack{\substack{e_{1}+\cdots+e_{d}+n=m\\
e_{1},\dots,e_{d}\geq0,n\geq1
}
}
}\zeta(n+1)\zeta(k_{1}+e_{1},\dots,k_{d}+e_{d}).
\end{align*}
\end{example}

The following is an equivalent formulation of Theorem \ref{thm:MainThm}:
\begin{thm}[Equivalent formulation of Theorem \ref{thm:MainThm}]
\label{thm:MainThm_equiv}For $w\in\mathfrak{h}^{0}$, we have
\begin{equation}
Z^{\shuffle}\circ\sigma\circ\tau\left(\frac{1}{1-xA}w\frac{1}{1-yB}\right)=Z^{\shuffle}\circ\sigma\left(\frac{1}{1-xA}w\frac{1}{1-yB}\right)\times\frac{\Gamma(1+A)\Gamma(1-T+B)}{\Gamma(1+B)\Gamma(1-T+A)}.\label{eq:mainthm_e1}
\end{equation}
\end{thm}

\section{Proof of the equivalence of Theorems \ref{thm:MainThm} and \ref{thm:MainThm_equiv}}

Put $\bar{\sigma}=\tau\circ\sigma\circ\tau$. Then since $Z_{Y,X}^{\shuffle}=Z_{X,Y}^{\shuffle}\circ\tau$,
we have $Z_{Y,X}^{\shuffle}\circ\sigma\circ\tau=Z_{X,Y}^{\shuffle}\circ\bar{\sigma}$,
which also implies that $Z^{\shuffle}\circ\sigma\circ\tau=Z^{\shuffle}\circ\bar{\sigma}$.
Note that any element of $\mathfrak{h}$ can be written as a linear
combination of the terms $x^{a}wy^{b}$ with $w\in\mathfrak{h}^{0}$
and $a,b\in\mathbb{Z}_{\geq0}$. Thus Theorem \ref{thm:MainThm} is
equivalent to the statement that
\begin{equation}
Z_{X,Y}^{\shuffle}\circ\bar{\sigma}\left(\frac{1}{1-xA}w\frac{1}{1-yB}\right)=\rho\circ Z_{X,Y}^{\shuffle}\circ\sigma\left(\frac{1}{1-xA}w\frac{1}{1-yB}\right)\label{eq:equiv_form1}
\end{equation}
for all $w\in\mathfrak{h}^{0}$. Define a linear map ${\rm reg}_{\shuffle}:\mathfrak{h}\to\mathfrak{h}^{0}$
by
\[
{\rm reg}_{\shuffle}(w\shuffle x^{a}\shuffle y^{b})=w\delta_{a,0}\delta_{b,0}\ \ \ (w\in\mathfrak{h}^{0}).
\]
Then $Z^{\shuffle}=Z\circ{\rm reg}_{\shuffle}$. For $a,b\in\mathbb{Z}_{\geq0}$,
define a linear map $D_{a,b}:\mathfrak{h}\to\mathfrak{h}$ by
\[
D_{a,b}(u_{1}\cdots u_{m})=\begin{cases}
u_{a+1}\cdots u_{m-b} & m\geq a+b,u_{1}=\cdots=u_{a}=x,u_{m-b+1}=\cdots=u_{m}=y\\
0 & {\rm otherwise},
\end{cases}
\]
where $u_{1},\dots,u_{m}\in\{x,y\}$. Then we have
\begin{equation}
w=\sum_{a=0}^{\infty}\sum_{b=0}^{\infty}{\rm reg}_{\shuffle}(D_{a,b}(w))\shuffle x^{a}\shuffle y^{b}\label{eq:shuffle_factorization}
\end{equation}
for $w\in\mathfrak{h}$ as a special case of the following lemma.
\begin{lem}[{\cite[Lemma  3.2.4 and equation (3.2.20)]{Panzer}}]
Let $\Sigma$ be a set, and $\mathcal{A}\coloneqq\mathbb{Q}\left\langle \Sigma\right\rangle $
the free non-commutative polynomial ring generated by $\Sigma$. Define
a linear map $c:\mathcal{A}\to\mathbb{Q}$ by $c(u_{1}\cdots u_{k})=\delta_{k,0}$,
where $u_{1},\dots,u_{k}\in\Sigma$. Let $A$ and $B$ be disjoint
subsets of $\Sigma$. Put
\[
\mathcal{A}^{0}\coloneqq\sum_{k=0}^{\infty}\sum_{\substack{u_{1},\dots,u_{k}\in\Sigma\\
u_{1}\notin B,u_{k}\notin A
}
}u_{1}\cdots u_{k}\mathbb{Q}\subset\mathcal{A}
\]
and define a linear map $s:\mathcal{A}^{0}\otimes\mathbb{Q}\left\langle A\right\rangle \otimes\mathbb{Q}\left\langle B\right\rangle \to\mathcal{A}$
by $s(w_{1}\otimes w_{2}\otimes w_{3})=w_{1}\shuffle w_{2}\shuffle w_{3}$.
Then $s$ is bijective. Furthermore, for $(a_{1},\dots,a_{r})\in A^{r}$,
$(b_{1},\dots,b_{l})\in B^{l}$, and $w\in\mathcal{A}^{0}$, we have
\[
b_{1}\cdots b_{l}wa_{r}\cdots a_{1}=\sum_{i=0}^{l}\sum_{j=0}^{r}b_{1}\cdots b_{i}\shuffle a_{j}\cdots a_{1}\shuffle{\rm reg}_{A}^{B}\left(b_{i+1}\cdots b_{l}wa_{r}\cdots a_{j+1}\right),
\]
where ${\rm reg}_{A}^{B}$ is the map from $\mathcal{A}$ to $\mathcal{A}^{0}$
defined by ${\rm reg}_{A}^{B}=\left({\rm id}\otimes c\otimes c\right)\circ s^{-1}$.
\end{lem}

By (\ref{eq:shuffle_factorization}), we have
\[
Z_{X,Y}^{\shuffle}(w)=\sum_{a=0}^{\infty}\sum_{b=0}^{\infty}\frac{X^{a}Y^{b}}{a!b!}Z^{\shuffle}(D_{a,b}(w))\ \ \ (w\in\mathfrak{h}).
\]
Thus the left-hand side of (\ref{eq:equiv_form1}) is
\begin{align*}
Z_{X,Y}^{\shuffle}\circ\bar{\sigma}\left(\frac{1}{1-xA}w\frac{1}{1-yB}\right) & =\sum_{a=0}^{\infty}\sum_{b=0}^{\infty}\frac{X^{a}Y^{b}}{a!b!}Z^{\shuffle}\circ D_{a,b}\circ\bar{\sigma}\left(\frac{1}{1-xA}w\frac{1}{1-yB}\right).
\end{align*}
Let us calculate $D_{a,b}\circ\bar{\sigma}(\frac{1}{1-xA}w\frac{1}{1-yB})$.
By definition,
\[
\bar{\sigma}\left(\frac{1}{1-xA}w\frac{1}{1-yB}\right)=\sum_{k=0}^{\infty}\sum_{l=0}^{\infty}\left(\frac{1}{1-yT}x\right)^{k}\bar{\sigma}(w)y^{l}A^{k}B^{l}.
\]
Since $\bar{\sigma}(w)\in\mathfrak{h}^{0}$, we have
\[
D_{a,b}\left(\left(\frac{1}{1-yT}x\right)^{k}\bar{\sigma}(w)y^{l}\right)=0
\]
if $k<a$ or $l<b$. Thus
\begin{align*}
D_{a,b}\circ\bar{\sigma}\left(\frac{1}{1-xA}w\frac{1}{1-yB}\right) & =\sum_{k=a}^{\infty}\sum_{l=b}^{\infty}D_{a,b}\left(\left(\frac{1}{1-yT}x\right)^{k}\bar{\sigma}(w)y^{l}\right)A^{k}B^{l}\\
 & =A^{a}B^{b}D_{a,b}\left(\left(\frac{1}{1-yT}x\right)^{a}\left(\sum_{k=0}^{\infty}\sum_{l=0}^{\infty}\left(\frac{1}{1-yT}x\right)^{k}\bar{\sigma}(w)y^{l}A^{k}B^{l}\right)y^{b}\right)\\
 & =A^{a}B^{b}D_{a,b}\left(\left(\frac{1}{1-yT}x\right)^{a}\bar{\sigma}\left(\frac{1}{1-xA}w\frac{1}{1-yB}\right)y^{b}\right)\\
 & =A^{a}B^{b}\bar{\sigma}\left(\frac{1}{1-xA}w\frac{1}{1-yB}\right)\ \ \ (\text{by the definition of \ensuremath{D_{a,b}}}).
\end{align*}
Therefore, we have
\begin{align*}
Z_{X,Y}^{\shuffle}\circ\bar{\sigma}\left(\frac{1}{1-xA}w\frac{1}{1-yB}\right) & =\sum_{a=0}^{\infty}\sum_{b=0}^{\infty}\frac{X^{a}Y^{b}}{a!b!}A^{a}B^{b}\times Z^{\shuffle}\circ\bar{\sigma}\left(\frac{1}{1-xA}w\frac{1}{1-yB}\right)\\
 & =e^{AX+BY}\times Z^{\shuffle}\circ\bar{\sigma}\left(\frac{1}{1-xA}w\frac{1}{1-yB}\right).
\end{align*}
Similarly the right-hand side of (\ref{eq:equiv_form1}) is
\begin{align*}
 & \rho\circ Z_{X,Y}^{\shuffle}\circ\sigma\left(\frac{1}{1-xA}w\frac{1}{1-yB}\right)\\
 & =\rho\left(e^{AX+BY}\times Z^{\shuffle}\circ\sigma\left(\frac{1}{1-xA}w\frac{1}{1-yB}\right)\right)\\
 & =e^{AX+BY}\times Z^{\shuffle}\circ\sigma\left(\frac{1}{1-xA}w\frac{1}{1-yB}\right)\times\frac{\Gamma(1+A)\Gamma(1-T+B)}{\Gamma(1+B)\Gamma(1-T+A)}.
\end{align*}
Thus (\ref{eq:equiv_form1}) is equivalent to (\ref{eq:mainthm_e1}).

\section{Symmetric harmonic product}

Put $\mathfrak{h}'\coloneqq y\mathfrak{h}\oplus x\mathfrak{h}$. Let
$e_{0}=x$, $e_{1}=-y$. Define the \emph{symmetric harmonic product}
$\mstuffle:\mathfrak{h}'\otimes\mathfrak{h}'\to\mathfrak{h}'$ by

\begin{align*}
e_{a}\mstuffle e_{b_{1}}\cdots e_{b_{n}} & =e_{b_{1}}\cdots e_{b_{n}}\mstuffle e_{a}=e_{ab_{1}}\cdots e_{ab_{n}},\\
e_{a}w_{1}\mstuffle e_{b}w_{2} & =e_{ab}(w_{1}\mstuffle e_{b}w_{2})+e_{ab}(e_{a}w_{1}\mstuffle w_{2})-e_{ab}e_{0}(w_{1}\mstuffle w_{2})\ \ \ (w_{1},w_{2}\in\mathfrak{h}'),
\end{align*}
where $a,b,b_{1},\dots,b_{n}\in\{0,1\}$, and by bilinearity. Then
$\mstuffle$ has a following combinatorial description which is similar
to the one for $*$ given in \cite{AlgDiff}:
\[
e_{a_{0}}\cdots e_{a_{m}}\mstuffle e_{b_{0}}\cdots e_{b_{n}}=\sum_{{\bf p}=(p_{0},\dots,p_{m+n})}(-1)^{\#\{i\,\mid\,p_{i}\in(1/2+\mathbb{Z})^{2}\}}e_{c(p_{0})}\cdots e_{c(p_{m+n})},
\]
where ${\bf p}=(p_{0},\dots,p_{m+n})$ runs over all sequences of
elements of $(\frac{1}{2}\mathbb{Z})^{2}$ such that $p_{0}=(0,0)$,
$p_{m+n}=(m,n)$, $p_{i+1}-p_{i}\in\{(0,1),(1,0),(\frac{1}{2},\frac{1}{2})\}$,
and $\{p_{i},p_{i+1}\}\not\subset(\frac{1}{2}+\mathbb{Z})^{2}$ for
all $0\leq i<m+n$, and $c(p_{i})$ is defined by
\[
c(p_{i})=\begin{cases}
a_{x}b_{y} & p_{i}=(x,y)\in\mathbb{Z}^{2},\\
0 & p_{i}\in(\frac{1}{2}+\mathbb{Z})^{2}.
\end{cases}
\]

By this combinatorial description, we can see that the symmetric harmonic
product $\mstuffle$ is compatible with the reversal operation, i.e.,
\begin{equation}
\overleftarrow{w_{1}}\mstuffle\overleftarrow{w_{2}}=\overleftarrow{w_{1}\mstuffle w_{2}}\ \ \ (w_{1},w_{2}\in\mathfrak{h}'),\label{eq:harm_reverse}
\end{equation}
where $\overleftarrow{w}$ is the reversal word of $w$. By definition,
we have
\begin{equation}
w_{1}e_{1}\mstuffle w_{2}e_{1}=(w_{1}*w_{2})e_{1}\ \ \ (w_{1},w_{2}\in\mathfrak{h}),\label{eq:add_e1}
\end{equation}
where $*:\mathfrak{h}\times\mathfrak{h}\to\mathfrak{h}$ is the usual
harmonic product defined in \cite{AlgDiff}. Note that by the correspondence
\[
yx^{k_{1}-1}\cdots yx^{k_{d}-1}\longleftrightarrow(k_{1},\dots,k_{d}),
\]
the product $*$ on $\mathbb{Q}+y\mathfrak{h}$ corresponds to the
harmonic product on indices (see \cite{AlgDiff} for details).

Thus we have
\begin{align}
xw_{1}\mstuffle w_{2} & =w_{1}\mstuffle xw_{2}=x(w_{1}\mstuffle w_{2})\ \ \ (w_{1},w_{2}\in\mathfrak{h}'),\label{eq:front_x}\\
yw_{1}\mstuffle yw_{2} & =y(w_{1}\mstuffle yw_{2})+y(yw_{1}\mstuffle w_{2})+yx(w_{1}\mstuffle w_{2})\ \ \ (w_{1},w_{2}\in\mathfrak{h}').\label{eq:front_y}
\end{align}
By (\ref{eq:harm_reverse}), (\ref{eq:front_x}), and (\ref{eq:front_y}),
we have
\begin{align}
w_{1}x\mstuffle w_{2} & =w_{1}\mstuffle w_{2}x=(w_{1}\mstuffle w_{2})x\ \ \ (w_{1},w_{2}\in\mathfrak{h}'),\label{eq:back_x}\\
w_{1}y\mstuffle w_{2}y & =(w_{1}\mstuffle w_{2}y)y+(w_{1}y\mstuffle w_{2})y+(w_{1}\mstuffle w_{2})xy\ \ \ (w_{1},w_{2}\in\mathfrak{h}').\label{eq:back_yy}
\end{align}
Furthermore, putting $z=x+y$, we also have
\begin{equation}
w_{1}z\mstuffle w_{2}z=-(w_{1}*w_{2})z\ \ \ (w_{1},w_{2}\in\mathfrak{h})\label{eq:add_z}
\end{equation}
since
\begin{align*}
w_{1}z\mstuffle w_{2}z & =w_{1}y\mstuffle w_{2}y+(w_{1}\mstuffle w_{2}y)x+(w_{1}z\mstuffle w_{2})x=(w_{1}\mstuffle w_{2}y)z+(w_{1}z\mstuffle w_{2})z\\
 & =R_{y}^{-1}(w_{1}y\mstuffle w_{2}y)z=-(w_{1}*w_{2})z,
\end{align*}
where $R_{y}:\mathfrak{h}\to\mathfrak{h}y$ is defined by $R_{y}(w)=wy$.

\section{\label{sec:Overview}Proof of the main theorem assuming a few propositions}

In this section, we give a proof of Theorem \ref{thm:MainThm_equiv},
assuming some propositions which will be proved in later sections.
Since $\mathfrak{h}^{0}\coloneqq\mathbb{Q}\oplus y\mathfrak{h}x$,
Theorem \ref{thm:MainThm_equiv} is reduced to the two claims
\begin{equation}
Z^{\shuffle}\circ\bar{\sigma}\left(\frac{1}{1-xA}\frac{1}{1-yB}\right)=Z^{\shuffle}\circ\sigma\left(\frac{1}{1-xA}\frac{1}{1-yB}\right)\times\frac{\Gamma(1+A)\Gamma(1-T+B)}{\Gamma(1+B)\Gamma(1-T+A)}\label{eq:main_const}
\end{equation}
and
\begin{equation}
\forall w\in y\mathfrak{h}x,\ \ Z^{\shuffle}\circ\bar{\sigma}\left(\frac{1}{1-xA}w\frac{1}{1-yB}\right)=Z^{\shuffle}\circ\sigma\left(\frac{1}{1-xA}w\frac{1}{1-yB}\right)\times\frac{\Gamma(1+A)\Gamma(1-T+B)}{\Gamma(1+B)\Gamma(1-T+A)}.\label{eq:main_yhx}
\end{equation}
Then (\ref{eq:main_yhx}) is equivalent to
\begin{equation}
\forall w\in y\mathfrak{h}x,\ \ Z^{\shuffle}\circ\bar{\sigma}\left(\frac{1}{1-zA}w\frac{1}{1-zB}\right)=Z^{\shuffle}\circ\sigma\left(\frac{1}{1-zA}w\frac{1}{1-zB}\right)\times\frac{\Gamma(1+A)\Gamma(1-T+B)}{\Gamma(1+B)\Gamma(1-T+A)}\label{eq:main_yhx_z}
\end{equation}
since
\begin{align*}
\left\{ \frac{1}{1-xA}w\frac{1}{1-yB}\Bigg|w\in y\hat{\mathfrak{h}}x\right\}  & =\{u\in\hat{\mathfrak{h}}\mid(1-xA)u(1-yB)\in y\hat{\mathfrak{h}}x\}\\
 & =\{u\in\hat{\mathfrak{h}}\mid(1-zA)u(1-zB)\in y\hat{\mathfrak{h}}x\}\\
 & =\left\{ \frac{1}{1-zA}w\frac{1}{1-zB}\Bigg|w\in y\hat{\mathfrak{h}}x\right\} ,
\end{align*}
where $\hat{\mathfrak{h}}=\mathfrak{h}[[A,B]]$.

Let $\varphi$ be the automorphism of $\mathfrak{h}$ defined by $\varphi(x)=z$,
$\varphi(y)=-y$. Note that $\varphi\circ\varphi={\rm id}$. Define
$H:\mathfrak{h}\to\mathbb{R}[[A,B]]$ by
\[
H(w)=Z^{\shuffle}\circ\varphi\left(\frac{1}{1-xA}w\frac{1}{1-xB}\right).
\]
Our main theorem can be proved by combining the following propositions
which will be proved in later sections:
\begin{prop}
\label{prop:tau_o_tau}For $w\in\mathfrak{h}$, we have
\end{prop}

\begin{equation}
\bar{\sigma}(w)=\sigma(w)+\varphi\left(\frac{yT}{1+yT}z\mstuffle\varphi(\sigma(w))\right).\label{eq:omw}
\end{equation}

\begin{prop}
\label{prop:H_prod}For $w_{1},w_{2}\in y\mathfrak{h}z$, we have
\[
H(w_{1}\mstuffle w_{2})=-H(w_{1})H(w_{2})\frac{\sin\pi(A-B)}{\pi}.
\]
\end{prop}

\begin{prop}
\label{prop:H_special}We have
\begin{equation}
-H\left(\frac{yT}{1+yT}z\right)=\frac{\pi}{\sin\pi(A-B)}\left(\frac{\Gamma(1+A)\Gamma(1-T+B)}{\Gamma(1+B)\Gamma(1-T+A)}-1\right).\label{eq:H_special}
\end{equation}
\end{prop}

Before proceeding to the proof of the main theorem, let us introduce
the following useful formulas which will sometimes be used in the
remainder of this paper.
\begin{lem}
\label{lem:basic_lemma}Put $\mathcal{R}\coloneqq\sum_{v\in\{x,y\}}\sum_{S\in\{T,A,B\}}\mathbb{Q}Sv$.
For $P,Q\in\mathcal{R}$, $w_{0},\dots,w_{m}\in\mathfrak{h}[[T,A,B]]$,
and $u_{1},\dots,u_{m}\in\{x,y\}$, we have
\[
\frac{1}{1-P}\shuffle w_{0}u_{1}w_{1}\cdots u_{m}w_{m}=\left(\frac{1}{1-P}\shuffle w_{0}\right)u_{1}\left(\frac{1}{1-P}\shuffle w_{1}\right)\cdots u_{m}\left(\frac{1}{1-P}\shuffle w_{m}\right)
\]
and
\[
\frac{1}{1-P}\shuffle\frac{1}{1-Q}=\frac{1}{1-P-Q}.
\]
\end{lem}

We omit the proof of this lemma since it is almost obvious.
\begin{proof}[Proof of Theorem \ref{thm:MainThm_equiv} assuming Propositions \ref{prop:tau_o_tau},
\ref{prop:H_prod}, and \ref{prop:H_special}]
It is enough to prove (\ref{eq:main_yhx_z}) and (\ref{eq:main_const}).

To prove (\ref{eq:main_yhx_z}), we take $w\in y\mathfrak{h}x$ and
set $u\coloneqq(1-zA)\sigma\big(\frac{1}{1-zA}w\frac{1}{1-zB}\big)(1-zB)$.
Then
\begin{equation}
\varphi(u)\in y\mathfrak{h}z[[T,A,B]]\label{eq:phi_u_in_yhz}
\end{equation}
since
\begin{align*}
u & =\sigma\left(\left(1-zA+yxAT\right)\frac{1}{1-zA}w\frac{1}{1-zB}(1-zB+yxBT)\right)\\
 & =\sigma\left(\left(1+yxT\frac{A}{1-zA}\right)w\left(1+yxT\frac{B}{1-zB}\right)\right)\\
 & \in\sigma\left(y\mathfrak{h}x[[T,A,B]]\right)\subset y\mathfrak{h}x[[T,A,B]].
\end{align*}
Now (\ref{eq:main_yhx_z}) follows from the following calculation:
\begin{align*}
 & Z^{\shuffle}\circ\bar{\sigma}\left(\frac{1}{1-zA}w\frac{1}{1-zB}\right)\\
 & =Z^{\shuffle}\circ\sigma\left(\frac{1}{1-zA}w\frac{1}{1-zB}\right)+Z^{\shuffle}\circ\varphi\left(\frac{yT}{1+yT}z\mstuffle\varphi\left(\sigma\left(\frac{1}{1-zA}w\frac{1}{1-zB}\right)\right)\right)\ \ \ (\text{by Proposition \ref{prop:tau_o_tau}})\\
 & =Z^{\shuffle}\circ\sigma\left(\frac{1}{1-zA}w\frac{1}{1-zB}\right)+Z^{\shuffle}\circ\varphi\left(\frac{yT}{1+yT}z\mstuffle\frac{1}{1-xA}\varphi(u)\frac{1}{1-xB}\right)\\
 & =Z^{\shuffle}\circ\sigma\left(\frac{1}{1-zA}w\frac{1}{1-zB}\right)+H\left(\frac{yT}{1+yT}z\mstuffle\varphi(u)\right)\ \ \ (\text{by (\ref{eq:front_x}) and (\ref{eq:back_x}) })\\
 & =Z^{\shuffle}\circ\sigma\left(\frac{1}{1-zA}w\frac{1}{1-zB}\right)-H\left(\frac{yT}{1+yT}z\right)H(\varphi(u))\frac{\sin\pi(A-B)}{\pi}\ \ \ (\text{by (\ref{eq:phi_u_in_yhz}) and Proposition \ref{prop:H_prod}})\\
 & =Z^{\shuffle}\circ\sigma\left(\frac{1}{1-zA}w\frac{1}{1-zB}\right)\left(1-H\left(\frac{yT}{1+yT}z\right)\frac{\sin\pi(A-B)}{\pi}\right)\\
 & =Z^{\shuffle}\circ\sigma\left(\frac{1}{1-zA}w\frac{1}{1-zB}\right)\cdot\frac{\Gamma(1+A)\Gamma(1-T+B)}{\Gamma(1+B)\Gamma(1-T+A)}\ \ \ (\text{by Proposition \ref{prop:H_special}}),
\end{align*}

We now prove (\ref{eq:main_const}). Since
\begin{align*}
 & \bar{\sigma}\left(\frac{1}{1-xA}\frac{1}{1-yB}\right)\\
 & =\frac{1}{1-\frac{1}{1-yT}xA}\frac{1}{1-yB}\\
 & =\frac{1}{1-yT-xA}(1-yT)\frac{1}{1-yB}\\
 & =\frac{1}{1-yT-xA}+\frac{1}{1-yT-xA}y\frac{B-T}{1-yB}\\
 & =\frac{1}{1-xA}\shuffle\frac{1}{1-yT}+\frac{1}{1-xA}\shuffle\frac{1}{1-yB}\shuffle\frac{1}{1+y(B-T)}y\frac{B-T}{1+xA}\ \ \ (\text{by Lemma \ref{lem:basic_lemma}}),
\end{align*}
we have
\begin{align*}
Z^{\shuffle}\circ\bar{\sigma}\left(\frac{1}{1-xA}\frac{1}{1-yB}\right) & =Z^{\shuffle}\left(1+\frac{1}{1+y(B-T)}y\frac{B-T}{1+xA}\right)\\
 & =1-A(B-T)Z^{\shuffle}\left(\frac{1}{1+y(B-T)}yx\frac{1}{1+xA}\right)\\
 & =1-\sum_{a,b\geq1}\zeta(\overbrace{1,\dots,1}^{a-1},b+1)(T-B)^{a}(-A)^{b}\\
 & =\frac{\Gamma(1+A)\Gamma(1-T+B)}{\Gamma(1+A+B-T)}.
\end{align*}
(see \cite[equation (10)]{BBB} for the last equality). Similarly,
we also have
\[
Z^{\shuffle}\circ\sigma\left(\frac{1}{1-xA}\frac{1}{1-yB}\right)=\frac{\Gamma(1+B)\Gamma(1-T+A)}{\Gamma(1+A+B-T)}.
\]
Thus (\ref{eq:main_const}) is also proved.
\end{proof}

\section{Proof of Proposition \ref{prop:tau_o_tau}}

In this section, we prove Proposition \ref{prop:tau_o_tau}. For $m\in\mathbb{Z}$,
define $\sigma_{m}:\mathfrak{h}\to\mathfrak{h}$ by $\sigma(w)=\sum_{m=0}^{\infty}T^{m}\sigma_{m}(w)$
and $\sigma_{n}(w)=0$ for $n<0$. Furthermore, we put $\bar{\sigma}_{m}\coloneqq\tau\circ\sigma_{m}\circ\tau$.
Let $\diamond:\mathfrak{h}\times\mathfrak{h}\to\mathfrak{h}$ be the
bilinear map defined in \cite{specific}. It satisfies $\varphi(w_{1})*\varphi(w_{2})=\varphi(w_{1}\diamond w_{2})$,
$yw_{1}\diamond yw_{2}=y(yw_{1}\diamond w_{2})-y(w_{1}\diamond xw_{2})$,
and $yw_{1}\diamond xw_{2}=x(yw_{1}\diamond w_{2})+y(w_{1}\diamond xw_{2})$.
For $n\in\mathbb{Z}$ and $w\in\mathfrak{h}$, we put $f_{n}(w)\coloneqq y^{n}\diamond w-(y^{n-1}\diamond w)y$,
where we mean $y^{m}=0$ for $m<0$, i.e., 
\[
f_{n}(w)=\begin{cases}
0 & n<0,\\
w & n=0,\\
y^{n}\diamond w-(y^{n-1}\diamond w)y & n\geq1.
\end{cases}
\]
For $n\geq1$, since
\[
f_{n}(w)=(-1)^{n}\varphi\left(y^{n}*\varphi(w)-(y^{n-1}*\varphi(w))y\right)
\]
and
\begin{align*}
 & \left(y^{n}*\varphi(w)-(y^{n-1}*\varphi(w))y\right)y\\
 & =-y^{n+1}\mstuffle\varphi(w)y+(y^{n}\mstuffle\varphi(w)y)y\ \ \ (\text{by (\ref{eq:add_e1})})\\
 & =-\left(y^{n+1}\mstuffle\varphi(w)\right)y-\left(y^{n}\mstuffle\varphi(w)\right)xy\ \ \ (\text{by (\ref{eq:back_yy})})\\
 & =-\left(y^{n}z\mstuffle\varphi(w)\right)y\ \ \ (\text{by (\ref{eq:back_x})}),
\end{align*}
we have
\begin{equation}
f_{n}(w)=(-1)^{n+1}\varphi\left(y^{n}z\mstuffle\varphi(w)\right).\label{eq:fn_harmonictilde}
\end{equation}

\begin{lem}
\label{lem:fn_recur}For $n\in\mathbb{Z}$ and $w\in\mathfrak{h}$,
we have 
\begin{align*}
f_{n}(xw) & =xf_{n}(w)+yf_{n-1}(xw),\\
f_{n}(yw) & =yf_{n}(w)-yf_{n-1}(xw).
\end{align*}
\end{lem}

\begin{proof}
The lemma follows from $yw_{1}\diamond xw_{2}=x(yw_{1}\diamond w_{2})+y(w_{1}\diamond xw_{2})$
and $yw_{1}\diamond yw_{2}=y(yw_{1}\diamond w_{2})-y(w_{1}\diamond xw_{2})$.
\end{proof}
\begin{lem}
\label{lem:fj_sum}For $n\geq0$ and $w\in\mathfrak{h}$, we have
\[
\sum_{j=0}^{n}f_{j}(yx^{n-j}w)=yf_{n}(w).
\]
\end{lem}

\begin{proof}
By Lemma \ref{lem:fn_recur}, we have
\[
\sum_{j=0}^{n}f_{j}(yx^{n-j}w)=y\sum_{j=0}^{n}f_{j}(x^{n-j}w)-y\sum_{j=1}^{n}f_{j-1}(x^{n-j+1}w)=yf_{n}(w).\qedhere
\]
\end{proof}
\begin{lem}
\label{lem:same_recurrence}For $m\geq0$ and $w\in\mathfrak{h}$,
we have
\begin{align*}
\bar{\sigma}_{m}(xw) & =y\bar{\sigma}_{m-1}(xw)+x\bar{\sigma}_{m}(w),\\
\bar{\sigma}_{m}(yw) & =y\bar{\sigma}_{m}(w),\\
\sum_{j=0}^{m}f_{j}(\sigma_{m-j}(xw)) & =y\sum_{j=0}^{m-1}f_{j}(\sigma_{m-1-j}(xw))+x\sum_{j=0}^{m}f_{j}(\sigma_{m-j}(w)),\\
\sum_{j=0}^{m}f_{j}(\sigma_{m-j}(yw)) & =y\sum_{j=0}^{m}f_{j}(\sigma_{m-j}(w)).
\end{align*}
\end{lem}

\begin{proof}
The first two identities are obvious. The last two identities follow
from the following calculation:
\begin{align*}
\sum_{j=0}^{m}f_{j}(\sigma_{m-j}(xw)) & =\sum_{j=0}^{m}f_{j}(x\sigma_{m-j}(w))\\
 & =x\sum_{j=0}^{m}f_{j}(\sigma_{m-j}(w))+y\sum_{j=1}^{m}f_{j-1}(x\sigma_{m-j}(w))\ \ \ (\text{by Lemma \ref{lem:fn_recur}}),\\
\sum_{j=0}^{m}f_{j}(\sigma_{m-j}(yw)) & =\sum_{j=0}^{m}\sum_{i=0}^{m-j}f_{j}(yx^{i}\sigma_{m-j-i}(w))\\
 & =\sum_{n=0}^{m}\sum_{j=0}^{n}f_{j}(yx^{n-j}\sigma_{m-n}(w))\ \ \ \ \ (i+j=n)\\
 & =\sum_{n=0}^{m}yf_{n}(\sigma_{m-n}(w))\ \ \ \ \ (\text{by Lemma \ref{lem:fj_sum}}).\qedhere
\end{align*}
\end{proof}
\begin{lem}
\label{lem:sigmabar_fj_sigma}For $m\geq0$ and $w\in\mathfrak{h}$,
we have
\[
\bar{\sigma}_{m}(w)=\sum_{j=0}^{m}f_{j}(\sigma_{m-j}(w)).
\]
\end{lem}

\begin{proof}
Let us prove this lemma by induction on $m$ and on the length of
$w$. The case $m=0$ or $w=1$ is obvious. Furthermore, both sides
satisfy the same recurrence formula by Lemma \ref{lem:same_recurrence}.
Thus the lemma is proved.
\end{proof}
Now, Proposition \ref{prop:tau_o_tau} follows from the following
calculation:
\begin{align*}
\bar{\sigma}(w) & =\sum_{m=0}^{\infty}T^{m}\cdot\bar{\sigma}_{m}(w)\\
 & =\sum_{m=0}^{\infty}T^{m}\sum_{j=0}^{m}f_{j}(\sigma_{m-j}(w))\ \ \ (\text{by Lemma \ref{lem:sigmabar_fj_sigma}})\\
 & =\sum_{j=0}^{\infty}\sum_{n=0}^{\infty}T^{n+j}f_{j}(\sigma_{n}(w))\ \ \ (n=m-j)\\
 & =\sum_{j=0}^{\infty}T^{j}f_{j}(\sigma(w))\\
 & =\sigma(w)-\sum_{j=1}^{\infty}(-T)^{j}\varphi\left(y^{j}z\mstuffle\varphi(\sigma(w))\right)\ \ \ (\text{by (\ref{eq:fn_harmonictilde})})\\
 & =\sigma(w)+\varphi\left(\frac{yT}{1+yT}z\mstuffle\varphi(\sigma(w))\right).
\end{align*}

\section{Proof of Proposition \ref{prop:H_prod}}

In this section, we prove Proposition \ref{prop:H_prod} by using
the Kawashima relation. For $n\geq1$, put $z_{n}\coloneqq yx^{n-1}$.
Define a bilinear map $\circledast:\mathfrak{h}y\mathfrak{h}\times\mathfrak{h}y\mathfrak{h}\to\mathfrak{h}y\mathfrak{h}$
by
\[
w_{1}z_{m}\circledast w_{2}z_{n}=(w_{1}*w_{2})z_{m+n}.
\]
For $w\in y\mathfrak{h}$ and indeterminate $T$, define $K(w;T)\in\mathbb{R}[[T]]$
by
\[
K(w;T)\coloneqq\sum_{m=1}^{\infty}Z(y^{m}\circledast\varphi(w))T^{m}.
\]
Then the Kawashima relation can be stated as follows:
\begin{thm}[Kawashima relation, \cite{Kawashima}]
\label{thm:Kawashima}For $w_{1},w_{2}\in y\mathfrak{h}$, we have
\[
K(w_{1}*w_{2};T)=K(w_{1};T)K(w_{2};T).
\]
\end{thm}

We put $\Gamma_{1}(T)\coloneqq\exp(\sum_{n=2}^{\infty}\frac{\zeta(n)}{n}(-T)^{n})\in\mathbb{R}[[T]]$
and $\mathfrak{h}^{1}=\mathbb{Q}\oplus y\mathfrak{h}$. Then $\mathfrak{h}^{1}$
is closed under the harmonic product $*$. Define $Z^{*}:\mathfrak{h}^{1}\to\mathbb{R}$
by the properties $Z^{*}(w_{1}*w_{2})=Z^{*}(w_{1})Z^{*}(w_{2})$,
$Z^{*}(w)=Z(w)$, and $Z^{*}(y)=0$, where $w_{1},w_{2}\in\mathfrak{h}^{1}$,
and $w\in\mathfrak{h}^{0}$. Then the regularization theorem can be
stated as follows:
\begin{thm}[Regularization theorem, \cite{IKZ}]
For $w\in\mathfrak{h}^{0}$, we have
\[
Z^{\shuffle}\left(w\frac{1}{1-yT}\right)=\Gamma_{1}(T)Z^{*}\left(w\frac{1}{1-yT}\right).
\]
\end{thm}

\begin{cor}
\label{cor:RegThm_another}Let $L\in\mathbb{Q}T$. If $w\in\mathfrak{h}^{0}[[T]](1-yL)^{-1}$,
then
\[
Z^{\shuffle}(w)=\Gamma_{1}(L)Z^{*}(w).
\]
\end{cor}

\begin{lem}
\label{lem:another_expr_KawashimaFunc}For $u\in y\mathfrak{h}$,
we have 
\[
K(u;T)=\frac{\sin(\pi T)}{\pi}Z^{\shuffle}\left(\varphi(u)x\frac{1}{1+zT}\right).
\]
\end{lem}

\begin{proof}
It is enough to consider the case where $\varphi(u)=vz_{k}$ for some
$v\in\mathfrak{h}^{1}$ and $k\in\mathbb{Z}_{\geq1}$. We have
\begin{align*}
K(u;T) & =TZ\left(\frac{1}{1-yT}y\circledast vz_{k}\right)\\
 & =TZ\left(\left(\frac{1}{1-yT}*v\right)z_{k+1}\right).
\end{align*}
Note that for any formal power series $w_{1}$, $w_{2}$, and $w_{3}$
of $T$ whose coefficients are in $\mathfrak{h}^{1}$, $\bigoplus_{m=1}^{\infty}\mathbb{Q}z_{m}$,
and $\mathfrak{h}$ respectively, we have
\begin{equation}
\frac{1}{1-yT}*w_{1}w_{2}w_{3}=\left(\frac{1}{1-yT}*w_{1}\right)w_{2}(1+xT)\left(\frac{1}{1-yT}*w_{3}\right)\label{eq:stuffle_fact}
\end{equation}
since the combinatorial description of the harmonic product between
indices shows
\[
y^{N}*w_{1}w_{2}w_{3}=\sum_{N_{1}+N_{2}=N}\left(y^{N_{1}}*w_{1}\right)w_{2}\left(y^{N_{2}}*w_{3}\right)+\sum_{N_{1}+1+N_{2}=N}\left(y^{N_{1}}*w_{1}\right)w_{2}x\left(y^{N_{2}}*w_{3}\right).
\]
Note that the harmonic inverse of $\frac{1}{1-yT}$ is given by
\[
1-yT+yzT^{2}-yz^{2}T^{2}+\cdots=1-yT(1+zT)^{-1}.
\]
Thus, by applying (\ref{eq:stuffle_fact}) to $w_{1}=v$, $w_{2}=z_{k+1}(1+xT)^{-1}$,
and $w_{3}=1-yT(1+zT)^{-1}$, we obtain
\[
\frac{1}{1-yT}*vz_{k+1}\frac{1}{1+zT}=\left(\frac{1}{1-yT}*v\right)z_{k+1},
\]
and thus
\begin{align*}
K(u;T) & =TZ\left(\frac{1}{1-yT}*vz_{k+1}\frac{1}{1+zT}\right)\\
 & =TZ^{*}\left(\frac{1}{1-yT}\right)Z^{*}\left(vz_{k+1}\frac{1}{1+zT}\right).
\end{align*}
Here since 
\[
\frac{1}{1-yT}\in\mathfrak{h}^{0}[[T]]\frac{1}{1-yT}\ \ \ \text{and}\ \ \ vz_{k+1}\frac{1}{1+zT}\in\mathfrak{h}^{0}[[T]]\frac{1}{1+yT},
\]
we can apply Corollary \ref{cor:RegThm_another}, and we have
\begin{align*}
K(u;T) & =T\Gamma_{1}(T)^{-1}Z^{\shuffle}\left(\frac{1}{1-yT}\right)\Gamma_{1}(-T)^{-1}Z^{\shuffle}\left(vz_{k+1}\frac{1}{1+zT}\right)\\
 & =\frac{\sin(\pi T)}{\pi}Z^{\shuffle}\left(\varphi(u)x\frac{1}{1+zT}\right).\qedhere
\end{align*}
\end{proof}
Define a linear map $\alpha_{A}:\mathfrak{h}\to\mathfrak{h}[[A]]$
by
\[
\alpha_{A}(1)=1,\ \alpha_{A}(vw)=v\left(\frac{1}{1+xA}\shuffle w\right)\ \ \ (v\in\{x,y\},w\in\mathfrak{h}).
\]

\begin{lem}
\label{lem:alpha_T_stuffle}$\alpha_{A}(w_{1}*w_{2})=\alpha_{A}(w_{1})*\alpha_{A}(w_{2})$
for $w_{1},w_{2}\in\mathfrak{h}^{1}$.
\end{lem}

\begin{proof}
We prove this lemma based on the theory of quasi-symmetric functions
in \cite[Theorem 2.2]{HofQSym}. Let $\mathfrak{P}$ be the inverse
limit of $\mathbb{Z}[t_{1}^{-1},\dots,t_{n}^{-1}]$ and define a linear
map $\psi:\mathfrak{h}^{1}\to\mathfrak{P}$ by $\psi(yx^{k_{1}-1}\cdots yx^{k_{d}-1})=\sum_{0<m_{1}<\cdots<m_{d}}t_{m_{1}}^{-k_{1}}\cdots t_{m_{d}}^{-k_{d}}$.
Then $\psi$ is injective and $\psi(w_{1}*w_{2})=\psi(w_{1})\psi(w_{2})$
for all $w_{1},w_{2}\in\mathfrak{h}^{1}$ (see \cite[Theorem 2.2]{HofQSym}).
Define the continuous ring automorphism $D_{A}$ of $\mathfrak{P}[[A]]$
by $D_{A}(A)=A$ and $D_{A}(t_{m}^{-1})=(t_{m}+A)^{-1}=\sum_{e=0}^{\infty}t_{m}^{-1-e}(-A)^{e}$.
Then for any $u=yx^{k_{1}-1}\cdots yx^{k_{d}-1}\in\mathfrak{h}^{1}$,
we have $\psi(\alpha_{A}(u))=D_{A}(\psi(u))$ since
\begin{align*}
\psi(\alpha_{A}(u)) & =\sum_{e_{1},\dots,e_{d}\geq0}\psi(yx^{k_{1}+e_{1}-1}\cdots yx^{k_{d}+e_{d}-1})\prod_{j=1}^{d}{k_{j}+e_{j}-1 \choose e_{j}}(-A)^{e_{j}}\\
 & =\sum_{0<m_{1}<\cdots<m_{d}}\prod_{j=1}^{d}\left(\sum_{e_{j}=0}^{\infty}t_{m_{j}}^{-k_{j}-e_{j}}{k_{j}+e_{j}-1 \choose e_{j}}(-A)^{e_{j}}\right)\\
 & =\sum_{0<m_{1}<\cdots<m_{d}}\prod_{j=1}^{d}(t_{m_{j}}+A)^{-k_{j}}=D_{A}(\psi(u)).
\end{align*}
Thus for any $w_{1},w_{2}\in\mathfrak{h}^{1}$, $\psi(\alpha_{A}(w_{1}*w_{2}))=\psi(\alpha_{A}(w_{1})*\alpha_{A}(w_{2}))$.
Since $\psi$ is injective, the lemma is proved.
\end{proof}
Now, we can show Proposition \ref{prop:H_prod}.

\begin{proof}[Proof of Proposition \ref{prop:H_prod}]
For any $w=w'z\in y\mathfrak{h}z$, we have
\begin{align*}
H(w) & =Z^{\shuffle}\circ\varphi\left(\frac{1}{1-xA}w'z\frac{1}{1-xB}\right)\\
 & =Z^{\shuffle}\circ\varphi\left(\frac{1}{1-xA}\shuffle\alpha_{A}(w')z\frac{1}{1+x(A-B)}\right)\ \ \ (\text{by Lemma \ref{lem:basic_lemma}})\\
 & =Z^{\shuffle}\circ\varphi\left(\alpha_{A}(w')z\frac{1}{1+x(A-B)}\right)\ \ \ (\text{since \ensuremath{\varphi} is a \ensuremath{\shuffle}-homomorphism})\\
 & =Z^{\shuffle}\left(\varphi(\alpha_{A}(w'))x\frac{1}{1+z(A-B)}\right)\\
 & =\frac{\pi}{\sin(\pi(A-B))}K(\alpha_{A}(w');A-B)\ \ \ (\text{by Lemma \ref{lem:another_expr_KawashimaFunc}}).
\end{align*}
Thus for $u=u'z,v=v'z\in y\mathfrak{h}z$, we have
\begin{align*}
H(u\mstuffle v) & =-H((u'*v')z)\ \ \ (\text{by (\ref{eq:add_z})})\\
 & =-\frac{\pi}{\sin(\pi(A-B))}K(\alpha_{A}(u'*v');A-B)\\
 & =-\frac{\pi}{\sin(\pi(A-B))}K(\alpha_{A}(u')*\alpha_{A}(v');A-B)\ \ \ (\text{by Lemma \ref{lem:alpha_T_stuffle}})\\
 & =-\frac{\pi}{\sin(\pi(A-B))}K(\alpha_{A}(u');A-B)K(\alpha_{A}(v');A-B)\ \ \ (\text{by Theorem \ref{thm:Kawashima}})\\
 & =-\frac{\sin(\pi(A-B))}{\pi}H(u)H(v).\qedhere
\end{align*}
\end{proof}

\section{Proof of Proposition \ref{prop:H_special}}

The purpose of this section is to prove Proposition \ref{prop:H_special}.
The left-hand side of Proposition \ref{prop:H_special} is equal to
\[
-H\left(\frac{yT}{1+yT}z\right)=Z^{\shuffle}\left(\frac{1}{1-zA}\frac{yT}{1-yT}x\frac{1}{1-zB}\right).
\]
Here, by Lemma \ref{lem:basic_lemma}, we have
\begin{align*}
 & \frac{1}{1-zA}y\frac{1}{1-yT}x\frac{1}{1-zB}\\
 & =\frac{1}{1-xA}\shuffle\frac{1}{1-yB}\shuffle\frac{1}{1-yA+yB}y\frac{1}{1+xA+yB-yT}x\frac{1}{1+xA-xB}.
\end{align*}
Thus
\begin{align*}
 & -H\left(\frac{yT}{1+yT}z\right)\\
 & =TZ\left(\frac{1}{1-yA+yB}y\frac{1}{1+xA+yB-yT}x\frac{1}{1+xA-xB}\right)\\
 & =(-\alpha+\beta-\gamma)Z\left(\frac{1}{1-y\alpha}y\frac{1}{1-x\gamma-y\beta}x\frac{1}{1+x\alpha}\right)\ \ \ (\alpha=A-B,\beta=T-B,\gamma=-A).
\end{align*}
Thus it suffices to prove the following identity.
\begin{prop}
We have
\begin{align*}
 & (-\alpha+\beta-\gamma)Z\left(\frac{1}{1-y\alpha}y\frac{1}{1-x\gamma-y\beta}x\frac{1}{1+x\alpha}\right)\\
 & =\frac{\pi}{\sin\pi\alpha}\left(\frac{\Gamma(1-\gamma)\Gamma(1-\beta)}{\Gamma(1-\alpha-\gamma)\Gamma(1+\alpha-\beta)}-1\right).
\end{align*}
\end{prop}

\begin{proof}
Note that the equality which we want to show is the one of formal
power series in $\alpha$, $\beta$, $\gamma$. By the iterated integral
expression of MZVs, we have
\begin{align*}
 & Z\left(\frac{1}{1-y\alpha}y\frac{1}{1-x\gamma-y\beta}x\frac{1}{1+x\alpha}\right)\\
 & =\sum_{k,l,m=0}^{\infty}Z\left(\left(\alpha y\right)^{k}y\left(\gamma x+\beta y\right)^{l}x(-\alpha x)^{m}\right)\\
 & =\sum_{k,l,m=0}^{\infty}\int_{0<t_{a}<t_{b}<1}\left(\frac{1}{k!}\left(\alpha\int_{0}^{t_{a}}\frac{ds}{1-s}\right)^{k}\right)\left(\frac{1}{l!}\left(\gamma\int_{t_{a}}^{t_{b}}\frac{ds}{s}+\beta\int_{t_{a}}^{t_{b}}\frac{ds}{1-s}\right)^{l}\right)\left(\frac{1}{m!}\left(-\alpha\int_{t_{b}}^{1}\frac{ds}{s}\right)^{m}\right)\frac{dt_{a}dt_{b}}{(1-t_{a})t_{b}}\\
 & =\int_{0<t_{a}<t_{b}<1}\left(\frac{1}{1-t_{a}}\right)^{\alpha}\left(\frac{t_{b}}{t_{a}}\right)^{\gamma}\left(\frac{1-t_{a}}{1-t_{b}}\right)^{\beta}\left(\frac{1}{t_{b}}\right)^{-\alpha}\frac{dt_{a}dt_{b}}{(1-t_{a})t_{b}}.
\end{align*}
Note that
\[
\int_{0<t_{a}<t_{b}<1}\left(\frac{1}{1-t_{a}}\right)^{\alpha}\left(\frac{t_{b}}{t_{a}}\right)^{\gamma}\left(\frac{1-t_{a}}{1-t_{b}}\right)^{\beta}\left(\frac{1}{t_{b}}\right)^{-\alpha}\frac{dt_{a}dt_{b}}{(1-t_{a})t_{b}}
\]
converges absolutely as a complex function of $\alpha$, $\beta$,
$\gamma$ in a neighborhood of the origin $(\alpha,\beta,\gamma)=(0,0,0)$.
Thus it is enough to show the equality
\begin{align*}
 & (-\alpha+\beta-\gamma)\int_{0<t_{a}<t_{b}<1}\left(\frac{1}{1-t_{a}}\right)^{\alpha}\left(\frac{t_{b}}{t_{a}}\right)^{\gamma}\left(\frac{1-t_{a}}{1-t_{b}}\right)^{\beta}\left(\frac{1}{t_{b}}\right)^{-\alpha}\frac{dt_{a}dt_{b}}{(1-t_{a})t_{b}}\\
 & =\frac{\pi}{\sin\pi\alpha}\left(\frac{\Gamma(1-\gamma)\Gamma(1-\beta)}{\Gamma(1-\alpha-\gamma)\Gamma(1+\alpha-\beta)}-1\right)
\end{align*}
as a complex function of $\alpha,\beta,\gamma$ in a neighborhood
of the origin. We denote by $(x)_{m}$ the Pochhammer symbol $x(x+1)\cdots(x+m-1)$.
Then, as a complex function, we have
\begin{align*}
 & \int_{0<t_{a}<t_{b}<1}\left(\frac{1}{1-t_{a}}\right)^{\alpha}\left(\frac{t_{b}}{t_{a}}\right)^{\gamma}\left(\frac{1-t_{a}}{1-t_{b}}\right)^{\beta}\left(\frac{1}{t_{b}}\right)^{-\alpha}\frac{dt_{a}dt_{b}}{(1-t_{a})t_{b}}\\
 & =\int_{0<t_{b}<1}t_{b}^{\alpha+\gamma-1}(1-t_{b})^{-\beta}\left(\int_{0<t_{a}<t_{b}}t_{a}^{-\gamma}(1-t_{a})^{-\alpha+\beta-1}dt_{a}\right)dt_{b}\\
 & =\int_{0<t_{b}<1}t_{b}^{\alpha+\gamma-1}(1-t_{b})^{-\beta}\left(\sum_{n=0}^{\infty}\int_{0<t_{a}<t_{b}}t_{a}^{n-\gamma}\frac{(\alpha-\beta+1)_{n}}{n!}dt_{a}\right)dt_{b}\\
 & =\int_{0<t_{b}<1}t_{b}^{\alpha+\gamma-1}(1-t_{b})^{-\beta}\left(\sum_{n=0}^{\infty}\frac{(\alpha-\beta+1)_{n}}{n!}\frac{t_{b}^{n-\gamma+1}}{n-\gamma+1}\right)dt_{b}\\
 & =\sum_{n=0}^{\infty}\frac{(\alpha-\beta+1)_{n}}{n!}\frac{1}{n-\gamma+1}\int_{0<t_{b}<1}t_{b}^{\alpha+n}(1-t_{b})^{-\beta}dt_{b}\\
 & =\sum_{n=0}^{\infty}\frac{(\alpha-\beta+1)_{n}}{n!}\frac{1}{n-\gamma+1}\frac{\Gamma(\alpha+n+1)\Gamma(-\beta+1)}{\Gamma(\alpha-\beta+n+2)}\\
 & =\frac{\Gamma(\alpha+1)\Gamma(-\beta+1)}{\Gamma(\alpha-\beta+1)}\sum_{n=0}^{\infty}\frac{(\alpha+1)_{n}}{n!(n-\gamma+1)(\alpha-\beta+n+1)}\\
 & =\frac{\Gamma(\alpha+1)\Gamma(-\beta+1)}{\Gamma(\alpha-\beta+1)(-\alpha+\beta-\gamma)}\sum_{n=0}^{\infty}\left(\frac{(\alpha+1)_{n}(\alpha-\beta+1)_{n}}{(\alpha-\beta+2)_{n}n!(\alpha-\beta+1)}-\frac{(\alpha+1)_{n}(-\gamma+1)_{n}}{(-\gamma+2)_{n}n!(-\gamma+1)}\right).
\end{align*}
Here, if $\Re(\alpha)<0$, then the infinite sums 
\[
\sum_{n=0}^{\infty}\frac{(\alpha+1)_{n}(\alpha-\beta+1)_{n}}{(\alpha-\beta+2)_{n}n!}\ \ \text{ and }\ \ \sum_{n=0}^{\infty}\frac{(\alpha+1)_{n}(-\gamma+1)_{n}}{(-\gamma+2)_{n}n!}
\]
converge to the hypergeometric series
\[
{_{2}F_{1}}\left[\begin{matrix}\alpha+1,\alpha-\beta+1\\
\alpha-\beta+2
\end{matrix};1\right]\ \ \text{ and }\ \ {_{2}F_{1}}\left[\begin{matrix}\alpha+1,-\gamma+1\\
-\gamma+2
\end{matrix};1\right]
\]
which are respectively equal to 
\[
\frac{\Gamma(\alpha-\beta+2)\Gamma(-\alpha)}{\Gamma(-\beta+1)}\ \ \text{ and }\ \ \frac{\Gamma(-\gamma+2)\Gamma(-\alpha)}{\Gamma(-\alpha-\gamma+1)}
\]
by Gauss's summation theorem. Thus
\begin{align*}
 & (-\alpha+\beta-\gamma)\int_{0<t_{a}<t_{b}<1}\left(\frac{1}{1-t_{a}}\right)^{\alpha}\left(\frac{t_{b}}{t_{a}}\right)^{\gamma}\left(\frac{1-t_{a}}{1-t_{b}}\right)^{\beta}\left(\frac{1}{t_{b}}\right)^{-\alpha}\frac{dt_{a}dt_{b}}{(1-t_{a})t_{b}}\\
 & =\frac{\Gamma(\alpha+1)\Gamma(-\beta+1)}{\Gamma(\alpha-\beta+1)}\left(\frac{\Gamma(\alpha-\beta+2)\Gamma(-\alpha)}{\Gamma(-\beta+1)(\alpha-\beta+1)}-\frac{\Gamma(-\gamma+2)\Gamma(-\alpha)}{\Gamma(-\alpha-\gamma+1)(-\gamma+1)}\right)\\
 & =\frac{\pi}{\sin(\pi\alpha)}\left(\frac{\Gamma(1-\gamma)\Gamma(1-\beta)}{\Gamma(1-\alpha-\gamma)\Gamma(1+\alpha-\beta)}-1\right)
\end{align*}
if $\Re(\alpha)<0$. By the identity theorem of complex functions,
this equality also holds without the condition $\Re(\alpha)<0$. Thus
the proposition is proved.
\end{proof}
Now all propositions in Section \ref{sec:Overview} are proved, and
hence Theorems \ref{thm:MainThm} and \ref{thm:MainThm_equiv} are
unconditionally proved.

\section*{Acknowledgements}

This work was supported by JSPS KAKENHI Grant Numbers JP18K03243 and
JP18K13392.

\end{document}